\documentclass{amsart}
\usepackage{enumitem}
\usepackage{pdfpages}

\newtheorem{theorem}{Theorem}[section]
\newtheorem{proposition}[theorem]{Proposition}
\newtheorem{lemma}[theorem]{Lemma}
\newtheorem{corollary}[theorem]{Corollary}
\newtheorem{example}[theorem]{Example}
\newtheorem{definition}[theorem]{Definition}
\numberwithin{equation}{section}

\newcommand{\1}[1]{{\tilde{#1}}}
\newcommand{\2}[1]{{\widetilde{#1}}}
\newcommand{\essup}{\operatorname{ess\,sup}}
\newcommand{\Int}{\operatorname{Int}}
\newcommand{\overbar}[1]{\mkern 4mu\overline{\mkern-4mu#1\mkern-1.5mu}\mkern 1.5mu}
\author{Mieczys{\l}aw Masty{\l}o}
\address{Faculty of Mathematics and Computer Science,
Adam Mickiewicz University;
Umultowska 87, 61-614 Pozna\'n, Poland}
\email{mastylo@amu.edu.pl}
\thanks{The first author was supported by the Foundation for Polish Science (FNP)}

\author{Gord Sinnamon}
\address{Department of Mathematics,
University of Western Ontario,
London, Canada}
\email{sinnamon@uwo.ca}
\thanks{The second author was supported by the Natural Sciences and Engineering Research Council of Canada}

\keywords{Calder\'on couple, Calder\'on-Mityagin couple, least decreasing majorant, level function}
\subjclass[2010]{Primary 46B70, Secondary 46E30, 46B42}

\begin{document}

\title[Couples related to decreasing functions]{Calder\'on-Mityagin couples of Banach spaces related to decreasing functions}

\begin{abstract}  A number of Calder\'on-Mityagin couples and relative Calder\'on-Mityagin pairs are identified among Banach function spaces defined in terms of the least decreasing majorant construction on the half line. The interpolation structure of such spaces is shown to closely parallel that of the rearrangement invariant spaces, and it is proved that a couple of these spaces is a Calder\'on-Mityagin couple if and only if the corresponding couple of rearrangement invariant spaces is a Calder\'on-Mityagin couple. Consequently, the class of all interpolation spaces for any couple of spaces of this type admits a complete description by the $K$-method if and only if the class of all interpolation spaces for the corresponding couple of rearrangement invariant spaces does. Analogous results are proved for spaces defined in terms of the level function construction.
In the main, the conclusions for both types of spaces remain valid when Lebesgue measure on the half line is replaced by a general Borel measure on $\mathbb R$. However, for certain measures the class of interpolation spaces of these new spaces may be degenerate, reducing the ``if and only if'' of the main results to a single implication.
\end{abstract}

\maketitle

\section{Introduction}
The problem of characterizing the class of all interpolation spaces for a given couple of compatible Banach spaces is one of the fundamental problems in interpolation theory. Although there are many constructions used to generate interpolation spaces, it can be very difficult to show that all interpolation spaces have been generated. For the couple $(L^1, L^\infty)$, however, the problem has been completely solved.  Mityagin \cite{M} in 1965 and, independently, Calder\'on \cite{C} in 1966 gave characterizations of the class of all interpolation spaces with respect to $(L^1, L^\infty)$. This discovery was an important point in the study of abstract interpolation spaces. 

Couples for which a similar characterization exists are called Calder\'on-Mityagin couples. While the benefits are great, the problem of determining whether or not a given Banach couple is a Calder\'on-Mityagin couple is still a difficult one in general. Nonetheless, many Calder\'on-Mityagin couples have been discovered; we refer to papers \cite{CNS} and \cite{Ka} and references therein related to this topic.

We concern ourselves here with two couples closely related to $(L^1, L^\infty)$. They are based on a pair of dual constructions, the least decreasing majorant and the level function, which arise naturally when monotone functions are involved and have been used to solve a variety of problems in the theory of function spaces and norm inequalities. For example, they were applied to describe the dual of the classical Lorentz spaces in \cite{Halp} and to give weight conditions for the boundedness of the Hardy operator in \cite{S1991, SS}. They have also been used to study absolute convergence of Fourier series, by Beurling in \cite{Be} and more recently in \cite{BLT}. See the references in \cite{BLT} for additional applications.

Spaces defined in terms of the least decreasing majorant construction form a class of subspaces of the interpolation spaces for $(L^1, L^\infty)$, while spaces defined in terms of the level function construction form a class of superspaces of the interpolation spaces for $(L^1, L^\infty)$. We will show that the parallels between these three classes of function spaces are very strong. Specifically, we show that a couple of interpolation spaces from one class is a Calder\'on-Mityagin couple if and only if the corresponding couple in either of the other two classes is also a Calder\'on-Mityagin couple.

The \emph{least decreasing majorant} of a Lebesgue measurable function $f$ on $(0,\infty)$, is defined by,
\[
\1f(x)=\essup\{|f(y)|: y\ge x\}.
\]
The map $f\mapsto\1f$ is sublinear, that is, $\2{f+g\,}\le\1f+\1g$. For any $f$, $\1f$ is non-negative and decreasing, $|f|\le\1f$ almost everywhere, and $\1f$ is minimal among functions with these two properties. In consequence, $\1f=f$ whenever $f$ is non-negative and decreasing; in particular,  the map is a projection, that is, $\1{\1f}=\1f$. If $X$ is a Banach function space of Lebesgue measurable functions on $(0,\infty)$, then $\2X$ is the space of functions $f$ for which $\1f\in X$, equipped with the norm $\|f\|_{\2X}=\|\1f\|_X$. Sublinearity makes it routine to verify that $\2X$ is a Banach function space.

The level function of $f$, denoted $f^o$, is a closely related projection defined (almost everywhere) to be the derivative of the least concave majorant of $\int_0^x |f|$, provided it is finite. Results of \cite{SLFD} show that that if $0\le f_n\uparrow f$ then $f_n^o\uparrow f^o$; this property is used to define $f^o$ in the event that the least concave majorant of $\int_0^x |f|$ is infinite for some $x>0$. For any $f$, $f^o$ is non-negative and decreasing, $f^o=f$ whenever $f$ is non-negative and decreasing, and $(f^o)^o=f^o$. If $X$ is an exact interpolation space for $(L^1,L^\infty)$, spaces of Lebesgue measurable functions on $(0,\infty)$, then $X^o$ is the space of functions $f$ for which $f^o\in X$, equipped with the norm $\|f\|_{X^o}=\|f^o\|_X$. Verification that $X^o$ is a Banach function space is routine, except for the triangle inequality.  If $f,g\in X^o$ then the function $x\mapsto\int_0^x(f^o+g^o)$ is a concave majorant of $x\mapsto\int_0^x|f+g|$ and hence $\int_0^x(f+g)^o\le\int_0^x(f^o+g^o)$ for all $x>0$. The hypothesis on $X$ and Calder\'on's results from \cite{C} show that
\[
\|f+g\|_{ X^o}=\|(f+g)^o\|_X\le\|f^o+g^o\|_X\le\| f^o\|_X+\| g^o\|_X=\|f\|_{ X^o}+\|g\|_{ X^o}.
\]

It is immediate that $\2{L^\infty}=L^\infty$ with identical norms and $(L^1)^o=L^1$ with identical norms, so in future we avoid writing $\2{L^\infty}$ and $ (L^1)^o$. On the other hand, it is clear that $\|f\|_{L^1}\le\|f\|_{\2{L^1}}$ whenever $f\in\2{L^1}$ and a calculation shows that,
\begin{equation}\label{DooNorm}
\|f\|_{(L^\infty)^o}=\sup_{x>0}\frac1x\int_0^x|f|\le\|f\|_{L^\infty}, \quad f\in L^\infty.
\end{equation}

To complete this introduction, we recall some fundamental definitions, notation and two key known results that will be used in the paper. For notation and background for Banach function spaces we refer to \cite{Z} and for interpolation we refer the reader to \cite{BS2} and \cite{BK}. A pair $(X_0,X_1)$ of Banach spaces that can be continuously embedded into a common Hausdorff topological vector space is called a \emph{compatible couple}. If $(X_0,X_1)$ and $(Y_0,Y_1)$ are compatible couples we write $T:(X_0,X_1)\to(Y_0,Y_1)$ to indicate that $T:X_0+X_1\to Y_0+Y_1$ is a linear operator and the restriction of $T$ to $X_j$ is a bounded map into $Y_j$ for $j=0,1$. The \emph{norm} of such an operator is the maximum of the norms of these two maps. A subspace $X$ of $X_0+X_1$, containing $X_0\cap X_1$, is called an \emph{exact interpolation space} for $(X_0,X_1)$ provided every linear operator $T:(X_0,X_1)\to(X_0,X_1)$ of norm one is also a bounded map from $X$ to $X$ of norm at most one. Let $\Int_1(X_0,X_1)$ denote the collection of all exact interpolation spaces for $(X_0,X_1)$.

The \emph{$K$-functional} of $x\in X_0+X_1$ is defined by,
\[
K(t,x;X_0,X_1)=\inf\{\|x_0\|_{X_0}+t\|x_1\|_{X_1}: x_0\in X_0, x_1\in X_1, x_0+x_1=x\}.
\]
A pair of compatible couples $[(X_0, X_1),(Y_0,Y_1)]$ form a \emph{$c$-uniform relative Calder\'on-Mityagin pair} if whenever $K(t,y;Y_0,Y_1)\le K(t,x;X_0,X_1)$ for all $t>0$,
there exists a $T:(X_0,X_1)\to(Y_0,Y_1)$, of norm at most $c$, such that $Tx=y$. If  $(X_0,X_1)=(Y_0,Y_1)$ we call $(X_0,X_1)$ a $c$-\emph{uniform Calder\'on-Mityagin couple.} If $c=1$ we call the pair (or the couple) \emph{exact}.

With these definitions, the famous description, due to Calder\'on \cite{C}, of all interpolation of spaces with respect to $(L^1, L^\infty)$  simply states: \emph{The couple $(L^1,L^\infty)$ is an exact Calder\'on-Mityagin couple.}

A \emph{parameter of the $K$-method} is a Banach function space $\Phi$ of functions on the measure space $((0,\infty),dt/t)$ that contains the function $t\mapsto\min(1,t)$. The \emph{$K$-method} with parameter $\Phi$ is the interpolation functor that takes the couple $(X_0,X_1)$ to the space $(X_0,X_1)_\Phi$ consisting of all $x\in X_0+X_1$ for which the norm,
\[
\|x\|_{(X_0,X_1)_\Phi}=\|K(\cdot,x; X_0,X_1)\|_\Phi,
\]
is finite. The \emph{$K$-divisibility constant} of a couple $(X_0,X_1)$ is the smallest $\gamma$ such that whenever $K(\cdot,x;X_0,X_1)\le\sum_{j=1}^\infty \varphi_j$, with $0\le\varphi_j$ concave on $(0,\infty)$, there exist $x_j$ such that $x=\sum_{j=1}^\infty x_j$ and $K(\cdot,x_j;X_0,X_1)\le\gamma\varphi_j$ for each $j$. The $K$-divisibility constant of $(L^1,L^\infty)$ is known to be $1$.

Brudny\u\i{} and Krugljak \cite[Theorem 4.4.5]{BK} showed that $\gamma<14$ for every couple, and used the $K$-method to give a complete description of all interpolation spaces for Calder\'on-Mityagin couples: {\it If $(X_0, X_1)$ is a $c$-uniform Calder\'on-Mityagin couple and $X\in\Int_1(X_0,X_1)$, then there exists a parameter of the $K$-method,  $\Phi$,
such that $X=(X_0,X_1)_\Phi$ and
\[
\|x\|_{(X_0,X_1)_\Phi} \leq \|x\|_X \leq c \gamma \|x\|_{(X_0,X_1)_\Phi}, \quad x\in X.
\]
Here $\gamma$ is the $K$-divisibility constant of $(X_0, X_1)$.}

In the special case that $(X_0,X_1)$ is an exact Calder\'on-Mityagin couple with $K$-divisibility constant $\gamma=1$, the conclusion is that $X=(X_0,X_1)_\Phi$ with identical norms. It is this case we will apply to get our main results in Section \ref{CCaRCP}.

\section{Operators Between Couples Close To $(L^1,L^\infty)$}\label{operators}

In this section we recall or construct a number of linear operators that will be needed to work with Calder\'on-Mityagin couples and relative Calder\'on-Mityagin pairs. Since $\2{L^\infty}=L^\infty$ and $ (L^1)^o=L^1$ we write the couple $(\2{L^1}, \2{L^\infty})$ as $(\2{L^1}, L^\infty)$ and the couple $( (L^1)^o,(L^\infty)^o)$ as $(L^1,(L^\infty)^o)$. Their $K$-functionals have been calculated in \cite[Theorem 1]{SInt} and  \cite[Lemma 2.1]{MS}.

\begin{proposition}\label{Kfnls} If $f\in \2{L^1}+L^\infty$ then $K(t,f;\2{L^1},L^\infty)=K(t,\1f;L^1,L^\infty)$ for $t>0$. If $f\in L^1+(L^\infty)^o$ then $K(t,f;L^1,(L^\infty)^o)=K(t,f^o;L^1,L^\infty)$.
\end{proposition}

Taking $t=1$ in this proposition shows that we have $(L^1+L^\infty\2{)\,}=\2{L^1}+L^\infty$ and $(L^1+L^\infty )^o=L^1+(L^\infty)^o$, with identical norms in both cases.

Our first operator does double duty, as a map from $(L^1,L^\infty)\to(\2{L^1}, L^\infty)$ and a map from $(L^1,(L^\infty)^o)$ to $(L^1,L^\infty)$.
\begin{lemma}\label S There is a linear map $S:L^1+(L^\infty)^o\to L^1+(L^\infty)^o$, such that $h\le Sh$ for any non-negative, decreasing $h$; and
\begin{enumerate}[leftmargin=1.85em, label=\rm{(\alph*)}]
\item\label{LooLoo} $S:L^\infty\to L^\infty$ has norm at most $1$,
\item\label{LoneLone} $S:L^1\to L^1$ has norm at most $2$,
\item\label{Lone2Lone} $S:L^1\to\2{L^1}$ has norm at most $4$,  and
\item\label{DooLoo} $S:(L^\infty)^o\to L^\infty$ has norm at most $2$.
\end{enumerate}
\end{lemma}
\begin{proof} For $h\in L^1+(L^\infty)^o$ let
$$
Sh=\sum_{j\in\mathbb Z} \frac{\chi_{[2^j,2^{j+1})}}{2^j-2^{j-1}}\int_{2^{j-1}}^{2^j}h.
$$
At each point the sum has at most one term, and $h$ is integrable on each interval $(2^j,2^{j+1})$, so $Sh$ well-defined for each $h\in L^1+(L^\infty)^o$. If $h$ is non-negative and decreasing, and $x>0$, then, with $j\in \mathbb Z$ chosen so that $x\in[2^j,2^{j+1})$,
\[
h(x)\le h(2^j)\le\frac1{2^j-2^{j-1}}\int_{2^{j-1}}^{2^j}h=Sh(x).
\]
The value of $Sh$ at each point is an average of $h$ over an interval so, if $h\in L^\infty$, we see $\|Sh\|_{L^\infty}\le\|h\|_{L^\infty}$, which proves \ref{LooLoo}. If $h\in L^1$ then $\|Sh\|_{L^1}\le2\|h\|_{L^1}$, giving \ref{LoneLone}. Also, $Sh$ is dominated by the decreasing function,
$$
\sum_{j\in\mathbb Z} \frac{\chi_{(0,2^{j+1})}}{2^j-2^{j-1}}\int_{2^{j-1}}^{2^j}|h|
$$
and so is its least decreasing majorant, $\2{Sh}$. Therefore, if $h\in L^1$,
$$
\|Sh\|_{\2{L^1}}=\|\2{Sh}\|_{L^1}\le\sum_{j\in\mathbb Z}\frac{2^{j+1}}{2^j-2^{j-1}}\int_{2^{j-1}}^{2^j}|h|=4\|h\|_{L^1}.
$$
This proves \ref{Lone2Lone}. If $h\in(L^\infty)^o$, then for each $j\in\mathbb Z$, the formula (\ref{DooNorm}) shows that
\[
\Big|\frac1{2^j-2^{j-1}}\int_{2^{j-1}}^{2^j}h\Big|\le \frac2{2^j}\int_0^{2^j}|h|\le2\|h\|_{(L^\infty)^o}.
\]
Thus $\|Sh\|_{L^\infty}\le2\|h\|_{(L^\infty)^o}$, giving \ref{DooLoo}.
\end{proof}

The remainder of the section is divided into two parts. Operators related to the couple $(\2{L^1},L^\infty)$ are considered first and operators related to $(L^1,(L^\infty)^o)$ follow.

\begin{theorem}\label{2Op} Suppose $f,g\in \2{L^1}+L^\infty$, $h\in L^1+L^\infty$, and
\[
K(t,g;\2{L^1},L^\infty)\le K(t,h;L^1,L^\infty)\le K(t,f;\2{L^1},L^\infty).
\]
\begin{enumerate}[leftmargin=1.85em, label=\rm{(\alph*)}]
\item\label{V1V2} There are linear operators $V_1,V_2:(L^1,L^\infty)\to(L^1,L^\infty)$, each of norm at most $1$, such that $V_1\1f=h$ and $V_2h=\1g$.
\item\label{2I}
The identity map $I:(\2{L^1},L^\infty)\to(L^1,L^\infty)$ has norm at most $1$.
\item\label{M}
There is a linear operator $\overbar M:(L^1,L^\infty)\to(\2{L^1},L^\infty)$, of norm at most $4$, such that $\overbar M\1g=g$.
\item\label{W} There are linear operators $W_1,W_2,W_3:(\2{L^1},L^\infty)\to(\2{L^1},L^\infty)$, each of norm at most $1$, such that $W_1f=\1f$, $W_2\1f=\1g$, and $W_3\1g=g$.
\end{enumerate}\end{theorem}
For reference, these operators may be represented schematically as follows.
\begin{equation}\label{2schema}
\lower 6ex\hbox{\includegraphics{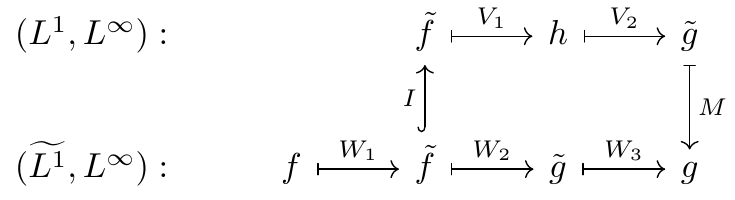}}
%\begin{tikzcd}
%(L^1,L^\infty):
%&&\1f \arrow[r,mapsto,"V_1"]
%&h \arrow[r,mapsto,"V_2"]
%&\1g \arrow[d,mapsto,"M"]\\
%(\2{L^1},L^\infty):
%&f \arrow[r,mapsto,"W_1"]
%&\1f \arrow[r,mapsto,"W_2"]\arrow[u,hook, "I"]
%&\1g \arrow[r,mapsto,"W_3"]
%&g
%\end{tikzcd}
\end{equation}
See Definition \ref{Mdef}, below, for the relationship between $M$ and $\overbar M$.
\begin{proof} By Proposition \ref{Kfnls}, we have
\[
K(t,\1g;L^1,L^\infty)\le K(t,h;L^1,L^\infty)\le K(t,\1f;L^1,L^\infty).
\]
Thus Calder\'on's theorem provides $V_1$ and $V_2$ to prove \ref{V1V2}. Since $|\psi|\le\1\psi$, \ref{2I} is trivial. For \ref M we define $\overbar M:(L^1,L^\infty)\to(\2{L^1},L^\infty)$ by
$\overbar M\psi=(g/S\1g)S\psi$, where $S$ is the operator of Lemma \ref S. Since $|g|\le\1g\le S\1g$, it follows from Lemma \ref S, parts \ref{Lone2Lone} and \ref{LooLoo}, that the norm of $\overbar M$ is at most $4$. (If necessary, take $0/0=0$ in the definition of $\overbar M$.)

Now we turn to \ref W. Let $W_3\psi=(g/\1g)\psi$, which has norm at most $1$ since $|g/\1g|\le1$. For $W_2$ we apply Theorem 5B of \cite{BS1}, which shows that there is a map $W_2:(L^1,L^\infty)\to(L^1,L^\infty)$, of norm at most $1$, such that $W_2\1f=\1g$ and which maps non-negative functions to non-negative functions and maps non-negative, decreasing functions to non-negative, decreasing functions. If $\psi\in \2{L^1}$ then $\1\psi\in L^1$ and $\1\psi\pm \psi\ge0$ so $\pm W_2\psi\le W_2\1\psi$. It follows that $|W_2\psi|\le W_2\1\psi$. But $\1\psi$ is non-negative and decreasing so $W_2\1\psi$ is a decreasing majorant of $|W_2\psi|$. Therefore, the least decreasing majorant $\2{W_2\psi}$ of $|W_2\psi|$ is no greater than $W_2\1\psi$. We have,
\[
\|W_2\psi\|_\2{L^1}=\|\2{W_2\psi}\|_{L^1}\le\|W_2\1\psi\|_{L^1}\le\|\1\psi\|_{L^1}=\|\psi\|_\2{L^1}.
\]
This shows that the norm of $W_2$ on $\2{L^1}$ is at most $1$ so $W_2$ may be viewed as a map from $(\2{L^1},L^\infty)$ to $(\2{L^1},L^\infty)$ of norm at most one.

The existence of $W_1$ follows from the Hahn-Banach-Kantorovich theorem, see Theorem 1.25 of \cite{AB}. On the one-dimensional subspace $\mathbb R f$ of $\2{L^1}+L^\infty$ the linear operator $W_1(\alpha f)=\alpha\1f\in \2{L^1}+L^\infty$ satisfies $W_1\psi\le\1\psi$. But $\psi\mapsto\1\psi$ is a sublinear map from $\2{L^1}+L^\infty$ to $\2{L^1}+L^\infty$ so $W_1$ extends to a linear operator from $\2{L^1}+L^\infty$ to $\2{L^1}+L^\infty$, preserving the bound $W_1\psi\le \1\psi$. Now, $-W_1\psi=W_1(-\psi)\le\2{-\psi}=\1\psi$ so we have $|W_1\psi|\le\1\psi$. Therefore, $\|W_1\psi\|_{L^\infty}\le\|\1\psi\|_{L^\infty}\le\|\psi\|_{L^\infty}$ and $\|W_1\psi\|_\2{L^1}\le\|\1\psi\|_\2{L^1}=\|\1{\1\psi}\|_{L^1}=\|\1\psi\|_{L^1}=\|\psi\|_\2{L^1}$. This shows $W_1$ has norm at most $1$.
\end{proof}

We do not claim that the operator $\overbar M$, constructed above, has smallest possible norm. Since this norm will appear in subsequent results we make the following definition.

\begin{definition}\label{Mdef} Let $m$ be a constant such that for each $g\in\2{L^1}+L^\infty$, there is a linear operator $M:(L^1,L^\infty)\to(\2{L^1},L^\infty)$, of norm at most $m$, satisfying $M\1g=g$.
\end{definition}
Theorem \ref{2Op}\ref{M} shows that $m=4$ is such a constant, with $M=\overbar M$. Lemma \ref{exM}, below, shows that no $m<9/8$ can satisfy Definition \ref{Mdef}.

\begin{lemma}\label{exM} There exists a function $g\in \2{L^1}\cap L^\infty$ such that any linear operator $M:(L^1,L^\infty)\to(\2{L^1},L^\infty)$ satisfying $M\1g=g$ has norm at least $9/8$.
\end{lemma}
\begin{proof}  Let $a=\chi_{[0,1)}$, $b=\chi_{[1,3)}$, and set $g=2a+b$. Clearly $g\in \2{L^1}\cap L^\infty$. Suppose $M:(L^1,L^\infty)\to(\2{L^1},L^\infty)$ has norm $C$ and satisfies $M\1g=g$. Since $g$ is decreasing, $g=\1g$ so $Mg=g$. Now $g=2$ on $(0,1)$, so
\[
2-\int_0^1Ma=\int_0^1Mg-Ma=\int_0^1M(a+b)\le\|M(a+b)\|_{L^\infty}\le C\|a+b\|_{L^\infty}=C.
\]
But $g=1$ on $(1,3)$, $Mb\le\2{Mb}$, and $\2{Mb}$ is non-negative and decreasing, so
\[
1-\int_1^3Ma=\frac12\int_1^3Mg-2Ma\le\frac12\int_1^3\2{Mb}
\le\frac13\int_0^3\2{Mb}\le\frac C3\int_0^\infty b=\frac23C.
\]
It follows that
\[
3-\frac53C\le\int_0^1Ma+\int_1^3Ma=\int_0^3Ma\le\int_0^3\2{Ma}
\le C\int_0^\infty a=C.
\]
Thus, $C\ge9/8$.
\end{proof}

A similar collection of operators can be found for the couple $(L^1,(L^\infty)^o)$.

\begin{theorem}\label{Op} Suppose $f,g\in L^1+(L^\infty)^o$, $h\in L^1+L^\infty$, and
\[
K(t,g;L^1,(L^\infty)^o)\le K(t,h;L^1,L^\infty)\le K(t,f;L^1,(L^\infty)^o).
\]
\begin{enumerate}[leftmargin=1.85em, label=\rm{(\alph*)}]
\item\label{V3V4} There are linear operators $V_3,V_4:(L^1,L^\infty)\to(L^1,L^\infty)$, each of norm at most $1$, such that $V_3f^o=h$ and $V_4h=g^o$.
\item\label{I} The identity map $I:(L^1,L^\infty)\to(L^1,(L^\infty)^o)$, has norm at most $1$.
\item\label{N} There is a linear operator $\overbar N:(L^1,(L^\infty)^o)\to(L^1,L^\infty)$, of norm at most $2$, such that $\overbar Nf=f^o$.
\item\label{U} There are linear operators $U_1,U_2,U_3:(L^1,(L^\infty)^o)\to(L^1,(L^\infty)^o)$, each of norm at most $1$, such that $U_1f=f^o$, $U_2f^o=g^o$, and $U_3g^o=g$.
\end{enumerate}
\end{theorem}
For reference, these operators may be represented schematically as follows.
\begin{equation}\label{schema}
\lower 6ex\hbox{\includegraphics{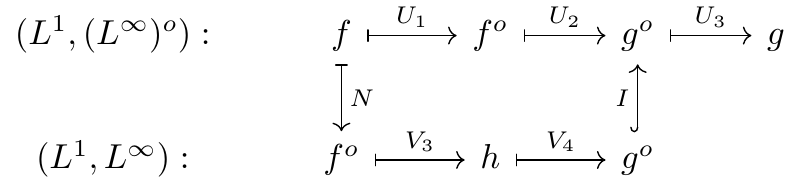}}
%\begin{tikzcd}
%(L^1,(L^\infty)^o):
%& f \arrow[r, mapsto,"U_1"] \arrow[d, mapsto,"N"]
%& f^o \arrow[r, mapsto, "U_2"]
%& g^o \arrow[r, mapsto, "U_3"]
%& g\\
%(L^1,L^\infty):
%& f^o \arrow[r, mapsto, "V_3"]
%& h\arrow[r, mapsto, "V_4"]
%& g^o\arrow[u, hook,"I"]
%\end{tikzcd}
\end{equation}
See Definition \ref{Ndef}, below, for the relationship between $N$ and $\overbar N$.
\begin{proof} By Proposition \ref{Kfnls}, we have
\[
K(t,g^o;L^1,L^\infty)\le K(t,h;L^1,L^\infty)\le K(t,f^o;L^1,L^\infty).
\]
So Calder\'on's theorem provides $V_3$ and $V_4$ to prove \ref{V3V4}. Part \ref{I} follows from (\ref{DooNorm}). The operators $U_1$ and $U_3$ were constructed in Lemmas 3.5 and 3.6 of \cite{MS}, respectively. Since $(f^o)^o=f^o$ and $(g^o)^o=g^o$ we have,
\[
K(t,g^o;L^1,(L^\infty)^o)\le K(t,f^o;L^1,(L^\infty)^o),
\]
so Theorem 3.8 of \cite{MS} gives the operator $U_2$.

For \ref N we define $\overbar N:(L^1,(L^\infty)^o)\to(L^1,L^\infty)$ by
$\overbar N\psi=(f^o/Sf^o)SU_1\psi$, where $S$ is the operator of Lemma \ref S. Since $f^o\le Sf^o$ and $U_1$ has norm at most $1$, it follows from Lemma \ref S, parts \ref{LoneLone} and \ref{DooLoo}, that the norm of $\overbar N$ is at most $2$. (If necessary, take $0/0=0$ in the definition of $\overbar N$.)
\end{proof}

We do not claim that the operator $\overbar N$, constructed above, has smallest possible norm. Since this norm will affect our subsequent results we make the following definition.

\begin{definition}\label{Ndef} Let $n$ be a constant such that for each $f\in L^1+(L^\infty)^o$, there is a linear operator $N:(L^1,(L^\infty)^o)\to(L^1,L^\infty)$, of norm at most $n$, satisfying $Nf= f^o$.
\end{definition}
Theorem \ref{Op}\ref{N} shows that $n=2$ is such a constant, with $N=\overbar N$. Lemma \ref{exN}, below, shows that no $n<9/8$ satisfies Definition \ref{Ndef}.

\begin{lemma}\label{exN}  There exists a function $f\in L^1\cap L^\infty$ such that any linear operator $N:(L^1,(L^\infty)^o)\to(L^1,L^\infty)$ satisfying $Nf=f^o$ has norm at least $9/8$.
\end{lemma}
\begin{proof} Let $a=\chi_{[0,1)}$, $b=\chi_{[1,3)}$, and set $f=2a+b$. Clearly, $f\in L^1\cap L^\infty$. Suppose $N:(L^1,(L^\infty)^o)\to(L^1,L^\infty)$ has norm $C$ and satisfies $Nf=f^o$. Since $f$ is decreasing, $f=f^o$ so $Nf=f$, that is, $N(2a+b)=2a+b$. Now $\|a+b\|_{(L^\infty)^o}=1$ and $\|a\|_{L^1}=1$, so
\[
\int_0^1N(a+b)+\int_0^3Na\le\|N(a+b)\|_{L^\infty}+\|Na\|_{L^1}
\le C\|a+b\|_{(L^\infty)^o}+C\|a\|_{L^1}=2C,
\]
and we have
\[
\int_1^3Na=\int_0^1N(a+b)+\int_0^3Na -\int_0^1N(2a+b)\le 2C-2.
\]
But, $\|Nb\|_{L^\infty}\le C\|b\|_{(L^\infty)^o}=(2/3)C$, by (\ref{DooNorm}), so
\[
2=\int_1^3N(2a+b)=2\int_1^3Na+\int_1^3Nb\le2(2C-2)+2(2/3)C=(16/3)C-4.
\]
Thus, $C\ge9/8$.
\end{proof}

\section{Concerning Calder\'on-Mityagin Couples and Relative Calder\'on-Mityagin Pairs}\label{CCaRCP}

The authors showed, in \cite{MS}, that $(L^1,(L^\infty)^o)$ is an exact Calder\'on-Mityagin couple and, more recently, Le\'snik, in \cite{L}, showed that $(\2{L^1},L^\infty)$ is also a Calder\'on-Mityagin couple. The first part of Theorem \ref{C2LL} provides an alternative proof of Le\'snik's result, improving his $1+\varepsilon$ constant to give exactness. More importantly, the pairs  $[(\2{L^1},L^\infty),(L^1,L^\infty)]$, $[(L^1,L^\infty),(\2{L^1},L^\infty)]$, $[(L^1,(L^\infty)^o),(L^1,L^\infty)]$, and $[(L^1,L^\infty),(L^1,(L^\infty)^o)]$ are shown to be relative Calder\'on-Mityagin pairs. See Theorems \ref{C2LL} and \ref{CLL}. 

Theorems \ref{2iff} and \ref{iff} set out the close correspondence between the class of all interpolation spaces for the couples $(L^1,L^\infty)$, $(\2{L^1},L^\infty)$, and $(L^1,(L^\infty)^o)$. Theorems \ref{2CC} and \ref{CC} take this correspondence further, showing that the relative Calder\'on-Mityagin pairs within $\Int_1(L^1,L^\infty)$ correspond exactly to the relative Calder\'on-Mityagin pairs within $\Int_1(\2{L^1},L^\infty)$ and also to those within $\Int_1(L^1,(L^\infty)^o)$.

Results for the couple $(\2{L^1},L^\infty)$ will be given first. Analogous results for the couple $(L^1,(L^\infty)^o)$ follow.

\begin{theorem}\label{C2LL} Let $m$ be a constant satisfying Definition {\rm \ref{Mdef}}. Then,
\begin{enumerate}[leftmargin=1.85em, label=\rm{(\alph*)}]
\item\label{2ECC} $(\2{L^1},L^\infty)$ is an exact Calder\'on-Mityagin couple.
\item\label{2ERC} $(\2{L^1},L^\infty)$ and $(L^1,L^\infty)$ form an exact relative Calder\'on-Mityagin pair.
\item\label{2URC} $(L^1,L^\infty)$ and $(\2{L^1},L^\infty)$ form an $m$-uniform relative Calder\'on-Mityagin pair.
\item\label{2notERC} $(L^1,L^\infty)$ and $(\2{L^1},L^\infty)$ do not form an exact relative Calder\'on-Mityagin pair.
\end{enumerate}
\end{theorem}
\begin{proof} Operator names refer to those constructed in Theorem \ref{2Op}. Refer to the diagram (\ref{2schema}).

To prove \ref{2ECC}, suppose $f,g\in\2{L^1}+L^\infty$ satisfy $K(t,g;\2{L^1},L^\infty)\le K(t,f;\2{L^1},L^\infty)$ for all $t>0$, and let $T=W_3W_2W_1$. All three have norm at most $1$ so $T:(\2{L^1},L^\infty)\to(\2{L^1},L^\infty)$ has norm at most $1$ as well. Also $Tf=W_3W_2W_1f=W_3W_2\1f=W_3\1g=g$.

For \ref{2ERC}, suppose $f\in\2{L^1}+L^\infty$ and $h\in L^1+L^\infty$ satisfy $K(t,h;L^1,L^\infty)\le K(t,f;\2{L^1},L^\infty)$ for all $t>0$, and let $T=V_1IW_1$. All three have norm at most $1$ so $T:(\2{L^1},L^\infty)\to(L^1,L^\infty)$ has norm at most $1$ as well. Also, $Tf=V_1IW_1f=V_1I\1f=V_1\1f=h$.

For \ref{2URC}, suppose $h\in L^1+L^\infty$ and $g\in\2{L^1}+L^\infty$ satisfy $K(t,g;\2{L^1},L^\infty)\le K(t,h;L^1,L^\infty)$ for all $t>0$, and let $T=MV_2$, where $M$ comes from Definition \ref{Mdef}. The operator $V_2$ has norm at most one and $M$ has norm at most $m$. Thus $T:(L^1,L^\infty)\to(\2{L^1},L^\infty)$ has norm at most $m$. Also, $Th=MV_2h=M\1g=g$.

Lemma \ref{exM} proves \ref{2notERC}.
\end{proof}

The next lemma will give control over the constants in subsequent results.

\begin{lemma}\label{Kdiv} The $K$-divisibility constant for the couple $(\2{L^1},L^\infty)$ is $1$.
\end{lemma}
\begin{proof} Suppose $g\in \2{L^1}+L^\infty$, $\{\varphi_j\}$ is a sequence of non-negative, concave functions such that $\sum_{j=0}^\infty\varphi_j(1)<\infty$, and
\[
K(t,g;\2{L^1},L^\infty)\le \sum_{j=1}^\infty\varphi_j(t),\quad t>0.
\]
Since $\int_0^t\1g=K(t,g;\2{L^1},L^\infty)<\infty$, the convex function $K=K(\cdot,g;\2{L^1},L^\infty)$ satisfies $K(0+)=0$. Replacing $\varphi_j$ by  $\min(K,\varphi_j)$ we may assume with no loss of generality that $\varphi_j(0+)=0$. Now, taking $f_j$ to be the derivative of $\varphi_j$ (which exists almost everywhere) we have $\varphi_j(t)=\int_0^t f_j$ for $t>0$. Note that a non-negative, concave function on $(0,\infty)$ is necessarily increasing, so each $f_j$ is non-negative and decreasing. Set $f=\sum_{j=0}^\infty f_j$ and observe that for each $J$, 
\[
f-\sum_{j=0}^J f_j=\sum_{j=J+1}^\infty f_j
\]
is a non-negative, decreasing function dominated by $f$. But $\int_0^1f\le\sum_{j=1}^\infty\varphi_j(1)<\infty$ so the dominated convergence theorem shows that,
$$
\Big\|f-\sum_{j=0}^Jf_j\Big\|_{\2{L^1}+L^\infty}=\int_0^1\Big(f-\sum_{j=0}^Jf_j\widetilde{\Big)}
=\int_0^1\Big(f-\sum_{j=0}^Jf_j\Big)\to0
$$
as $J\to\infty$. Thus $\sum_{j=0}^\infty f_j$ converges to $f$ in $\2{L^1}+L^\infty$.

Since $f$ is non-negative and decreasing,
\[
K(t,g;\2{L^1},L^\infty)\le \sum_{j=1}^\infty\varphi_j(t)=\int_0^tf=K(t,f;\2{L^1},L^\infty).
\]
Theorem \ref{2Op}\ref{W} provides operators $W_2$ and $W_3$ on $\2{L^1}+L^\infty$, both of norm at most $1$, such that $g=W_3W_2\tilde f=W_3W_2f$. Let $g_j=W_3W_2f_j$. The boundedness of $W_3W_2$ implies that $\sum_{j=1}^\infty g_j$ converges to $g$ in $\2{L^1}+L^\infty$. Because the norm of $W_3W_2$ is at most $1$, 
\[
K(t,g_j,\2{L^1},L^\infty)\le K(t,f_j,\2{L^1},L^\infty)=\int_0^t\1f_j=\int_0^tf_j=\varphi_j(t).\qedhere
\]
\end{proof}

\begin{theorem}\label{2iff} If $X$ is in $\Int_1(L^1,L^\infty)$ then $\2X$ is in $\Int_1(\2{L^1},L^\infty)$. Conversely, if $Y$ is in $\Int_1(\2{L^1},L^\infty)$ then $Y=\2X$, with identical norms, for some space $X$ in $\Int_1(L^1,L^\infty)$.
\end{theorem}
\begin{proof} If $X\in\Int_1(L^1,L^\infty)$ we apply the theorems of Calder\'on and Brudny\u\i-Krugljak to get a parameter $\Phi$ of the $K$-method such that $X=(L^1,L^\infty)_{\Phi}$, with identical norms. By Proposition \ref{Kfnls}, for any $f\in \2{L^1}+L^\infty$, all five of the statements: $f\in \2X$, $\1f\in X$, $K(\cdot,\1f;L^1,L^\infty)\in\Phi$, $K(\cdot,f;\2{L^1},L^\infty)\in\Phi$ and $f\in(\2{L^1},L^\infty)_\Phi$ are equivalent. Moreover,
\[
\|f\|_{\2X}=\|\1f\|_X=\|K(\cdot,\1f;L^1,L^\infty)\|_\Phi
=\|K(\cdot,f;\2{L^1},L^\infty)\|_\Phi=\|f\|_{(\2{L^1},L^\infty)_\Phi}.
\]
This shows $\2X=(\2{L^1},L^\infty)_\Phi$, with identical norms. In particular, $\2X\in \Int_1(\2{L^1},L^\infty)$. 

On the other hand, if $Y\in\Int_1(\2{L^1},L^\infty)$ then by Theorem \ref{C2LL}\ref{2ECC}, the Brudny\u\i-Krugljak theorem, and Lemma \ref{Kdiv}, there exists a parameter $\Phi$ of the $K$-method such that $Y=(\2{L^1},L^\infty)_{\Phi}$, with identical norms. Let $X=(L^1,L^\infty)_{\Phi}$. Then $X\in\Int_1(L^1,L^\infty)$. Arguing as above we see that $\2X=(\2{L^1},L^\infty)_\Phi=Y$, with identical norms.
\end{proof}

\begin{lemma}\label{2embed} Let $m$ be a constant satisfying Definition {\rm \ref{Mdef}}. Suppose $X_0$ and $X_1$ are in $\Int_1(L^1,L^\infty)$ and $\mathcal F$ is an exact interpolation functor. Set $X=\mathcal F(X_0,X_1)$.
\begin{enumerate}[leftmargin=1.85em, label=\rm{(\alph*)}]
\item\label{2cons} There exist $\Phi_0$ and $\Phi_1$, parameters of the $K$-method, such that $X_j=(L^1,L^\infty)_{\Phi_j}$ and $\2{X_j}=(\2{L^1},L^\infty)_{\Phi_j}$, in each case with identical norms, for $j=0,1$.
\item\label{2inX} $\mathcal F(\2{X_0},\2{X_1})\hookrightarrow\2X$ with norm at most $1$.
\item\label{2Xin} $\2X\hookrightarrow \mathcal F(\2{X_0},\2{X_1})$ with norm at most $m$.
\item\label{2approxK} $\2{X_0}+\2{X_1}=(X_0+X_1\2{)\,}$, and for all $f$ in this space, 
\[
K(t,\1f;X_0,X_1)\le K(t,f;\2{X_0},\2{X_1})\le mK(t,\1f;X_0,X_1),\quad t>0.
\]
\end{enumerate}
\end{lemma}
\begin{proof} The proof of part \ref{2cons} is contained in the proof of Theorem \ref{2iff}.
For part \ref{2inX}, fix $f\in\mathcal F(\2{X_0},\2{X_1})$. Then $f\in\2{X_0}+\2{X_1}\subseteq \2{L^1}+L^\infty$ so Theorem \ref{2Op} provides operators $I$ and $W_1$ such that $IW_1:(\2{L^1},L^\infty)\to(L^1,L^\infty)$ has norm at most $1$, and $IW_1f=\1f$. Real interpolation with parameters $\Phi_0$ and $\Phi_1$ shows that $IW_1:(\2{X_0},\2{X_1})\to(X_0,X_1)$, with norm at most $1$, and we may apply the exact functor $\mathcal F$ to obtain $IW_1:\mathcal F(\2{X_0},\2{X_1})\to X$, with norm at most $1$. We conclude that $\1f=IW_1f\in X$ and $\|f\|_{\2X}=\|\1f\|_X\le \|f\|_{F(\2{X_0},\2{X_1})}$. Thus, each $f\in\mathcal F(\2{X_0},\2{X_1})$ is in $\2X$ and the norm of the embedding is at most $1$.

For part \ref{2Xin}, fix $g\in\2X$. Since $\1g\in X=\mathcal F(X_0, X_1)\subseteq L^1+L^\infty$, $g\in\2{L^1}+L^\infty$. By Definition \ref{Mdef} there exists a linear operator $M:(L^1,L^\infty)\to(\2{L^1},L^\infty)$, of norm at most $m$, such that $M\1g=g$. It follows by real interpolation with parameter $\Phi_j$, that $M: X_j\to\2{X_j}$ with norm at most $m$, for $j=0,1$. Applying the exact functor $\mathcal F$ shows $M:X\to F(\2{X_0},\2{X_1})$, with norm at most $m$. Since $\1g\in X$, $g=M\1g\in F(\2{X_0},\2{X_1})$ and $\|g\|_{\mathcal F(\2{X_0},\2{X_1})}=\|M\1g\|_{\mathcal F(\2{X_0},\2{X_1})}\le m\|\1g\|_X=m\|g\|_\1X$. Thus, each $g\in\2X$ is in $\mathcal F(\2{X_0},\2{X_1})$ and the norm of the embedding is at most $m$.

For part \ref{2approxK} we apply \ref{2inX} and \ref{2Xin} to a family of exact functors. Fix $t>0$. If $(Z_0,Z_1)$ is a compatible couple of Banach spaces let $\Sigma_t(Z_0,Z_1)$ be the space $Z_0+Z_1$, equipped with the norm $\|f\|_{\Sigma_t(Z_0,Y_1)}=K(t,f;Z_0,Z_1)$. Then the embeddings from \ref{2inX} and \ref{2Xin} show that $\2{X_0}+\2{X_1}=(X_0+X_1\2{)\,}$, and for all $f$ in this space, the estimates of the constants in \ref{2inX} and \ref{2Xin} show
\[
K(t,\1f;X_0,X_1)=\|\1f\|_{\Sigma_t(X_0,X_1)}\le\|f\|_{\Sigma_t(\2{X_0},\2{X_1})}=K(t,f;\2{X_0},\2{X_1}),
\]
and
\[
K(t,f;\2{X_0},\2{X_1})=\|f\|_{\Sigma_t(\2{X_0},\2{X_1})}\le m\|\1f\|_{\Sigma_t(X_0,X_1)}=mK(t,\1f;X_0,X_1).\qedhere
\]
\end{proof}

Note that if $\mathcal F$ is a positive exact interpolation functor, then the assumption $X_0, X_1\in \Int_1(L^1,L^\infty)$ is not needed to prove \ref{2inX} above. Because $f\mapsto\2f$ is sublinear, the embedding $\mathcal F(\2{X_0},\2{X_1})\hookrightarrow\2X$, with norm at most $1$, is a consequence of Theorem 2.1 of \cite{M1}.

The next theorem gives the correspondence between relative Calder\'on-Mityagin pairs in $\Int_1(L^1,L^\infty)$ and those in $\Int_1(\2{L^1},L^\infty)$. Note that the converse below is stated only for couples in $\Int_1(\2{L^1},L^\infty)$ of the form $(\2{X_0},\2{X_1})$ with $X_0,X_1\in\Int_1(L^1,L^\infty)$. However, in view of Theorem \ref{2iff} there is no loss of generality in this formulation.
\begin{theorem}\label{2CC} Let $m$ be a constant satisfying Definition {\rm \ref{Mdef}} and let $X_0, X_1, Y_0, Y_1\in \Int_1(L^1,L^\infty)$. If $(X_0,X_1)$ and $(Y_0,Y_1)$ form a $c$-uniform relative Calder\'on-Mityagin pair then $(\2{X_0},\2{X_1})$ and $(\2{Y_0},\2{Y_1})$ form an $(m^2c)$-uniform relative Calder\'on-Mityagin pair. Conversely, if $(\2{X_0},\2{X_1})$ and $(\2{Y_0},\2{Y_1})$ form a $c$-uniform relative Calder\'on-Mityagin pair then $(X_0,X_1)$ and $(Y_0,Y_1)$ form an $(m^2c)$-uniform relative Calder\'on-Mityagin pair.
\end{theorem}
\begin{proof} Suppose $(X_0,X_1)$ and $(Y_0,Y_1)$ form a $c$-uniform relative Calder\'on-Mityagin pair, $f\in\2{X_0}+\2{X_1}$, $g\in\2{Y_0}+\2{Y_1}$, and $K(t,g;\2{Y_0},\2{Y_1})\le K(t,f;\2{X_0},\2{X_1})$ for $t>0$.  By Lemma \ref{2embed}\ref{2approxK}, applied first to the couple $(Y_0, Y_1)$ and then to the couple $(X_0,X_1)$, and by the homogeneity of the $K$-functional,
\[
K(t,\1g;Y_0,Y_1)\le K(t,g;\2{Y_0},\2{Y_1})
\le K(t,f;\2{X_0},\2{X_1})\le K(t,m\1f;X_0,X_1).
\]
By hypothesis, there exists a linear operator $T:(X_0,X_1)\to(Y_0,Y_1)$, of norm at most $c$, such that $T(m\1f)=\1g$. Since $f\in\2{X_0}+\2{X_1}$, $\1f\in X_0+X_1\subseteq L^1+L^\infty$. The product $IW_1:(\2{L^1},L^\infty)\to(L^1,L^\infty)$, from Theorem \ref{2Op}, has norm at most $1$ and takes $f$ to $\1f$. The operator $M:(L^1,L^\infty)\to(\2{L^1},L^\infty)$, from Definition \ref{Mdef}, has norm at most $m$ and takes $\1g$ to $g$. By Lemma \ref{2embed}\ref{2cons}, $IW_1:(\2X_0,\2X_1)\to(X_0,X_1)$ has norm at most $1$, and $M:(Y_0,Y_1)\to(\2Y_0,\2Y_1)$ has norm at most $m$. It follows that
\[
\includegraphics{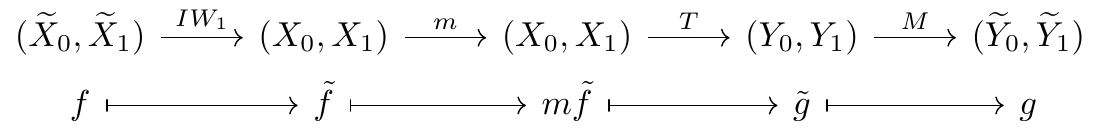}
%\begin{tikzcd}[row sep=.1ex]
%(\2X_0,\2X_1)\arrow[r,"IW_1"]
%&(X_0,X_1)\arrow[r,"m"]
%&(X_0,X_1)\arrow[r,"T"]
%&(Y_0,Y_1)\arrow[r,"M"]
%&(\2Y_0,\2Y_1)\\
%f\arrow[r,mapsto]
%&\1f\arrow[r,mapsto]
%&m\1f\arrow[r,mapsto]
%&\1g\arrow[r,mapsto]
%&g
%\end{tikzcd}
\]
where $m$ is used to denote multiplication by $m$. That is, the operator $\bar T=MTmIW_1$ maps $(\2X_0,\2X_1)$ to $(\2Y_0,\2Y_1)$, with norm at most $m^2c$, and $\bar Tf=g$. We conclude that $(\2X_0,\2X_1)$ and $(\2Y_0,\2Y_1)$ form an $(m^2c)$-uniform relative Calder\'on-Mityagin pair.

Conversely, suppose $(\2{X_0},\2{X_1})$ and $(\2{Y_0},\2{Y_1})$ form a $c$-uniform relative Calder\'on-Mityagin pair, $f\in X_0+X_1$, $g\in Y_0+Y_1$,  and $K(t,g;Y_0,Y_1)\le K(t,f;X_0,X_1)$ for $t>0$. Since $g$ and $g^*$ are equimeasurable, and $f$ and $f^*$ are equimeasurable,
\[
K(t,g^*;Y_0,Y_1)=K(t,g;Y_0,Y_1)\le K(t,f;X_0,X_1)=K(t,f^*,X_0,X_1).
\]
But $g^*=\2{\,g^*}$ and $f^*=\2{f^*}$, so by Lemma \ref{2embed}\ref{2approxK} and the homogeneity of the $K$-functional,
\[
K(t,g^*;\2{Y_0},\2{Y_1})\le mK(t,\2{\,g^*};Y_0,Y_1)
\le mK(t,\2{f^*},X_0,X_1)\le K(t,mf^*,\2{X_0},\2{X_1}).
\]
By the hypothesis, there exists a $T:(\2{X_0},\2{X_1})\to(\2{Y_0},\2{Y_1})$, with norm at most $c$ such that $T(mf^*)=g^*$.

Since $K(t,f^*;L^1,L^\infty)\le K(t,f;L^1,L^\infty)$ and $K(t,g;L^1,L^\infty)\le K(t,g^*;L^1,L^\infty)$, Calder\'on's theorem provides $V_f, V_g:(L^1,L^\infty)\to(L^1,L^\infty)$, each with norm at most $1$, such that $V_ff=f^*$ and $V_gg^*=g$. Exact interpolation shows that $V_f:(X_0,X_1)\to(X_0,X_1)$ and $V_g:(Y_0,Y_1)\to(Y_0,Y_1)$, each with norm at most $1$.

By Definition \ref{Mdef} there exists an $M:(L^1,L^\infty)\to(\2{L^1},L^\infty)$, with norm at most $m$, such that $M\2{f^*}=f^*$, that is, $Mf^*=f^*$. In Theorem \ref{2Op}\ref{2I} it was observed that the identity operator $I$ maps $(\2{L^1},L^\infty)$ to $(L^1,L^\infty)$, with norm at most $1$. By Lemma \ref{2embed}\ref{2cons}, $M:(X_0,X_1)\to(\2{X_0},\2{X_1})$, with norm at most $m$, and $I:(\2{Y_0},\2{Y_1})\to(Y_0,Y_1)$, with norm at most $1$. Now,
\[
\includegraphics{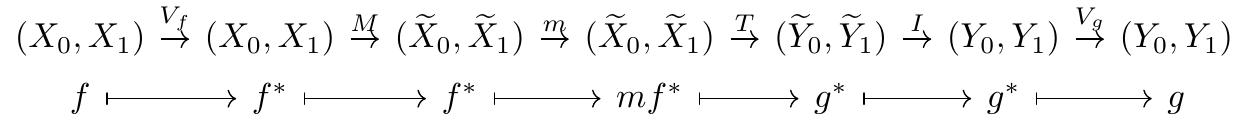}
%\begin{tikzcd}[row sep=.1ex, column sep=.85em]
%(X_0,X_1)\arrow[r,"V_f"]
%&(X_0,X_1)\arrow[r,"M"]
%&(\2X_0,\2X_1)\arrow[r,"m"]
%&(\2X_0,\2X_1)\arrow[r,"T"]
%&(\2Y_0,\2Y_1)\arrow[r,"I"]
%&(Y_0,Y_1)\arrow[r,"V_g"]
%&(Y_0,Y_1)\\
%f\arrow[r,mapsto]
%&f^*\arrow[r,mapsto]
%&f^*\arrow[r,mapsto]
%&mf^*\arrow[r,mapsto]
%&g^*\arrow[r,mapsto]
%&g^*\arrow[r,mapsto]
%&g
%\end{tikzcd}
\]
That is, with $\bar T=V_gITmMV_f$, $\bar T:(X_0,X_1)\to(Y_0,Y_1)$, with norm at most $m^2c$, and $\bar Tf=g$. This shows that the couples $(X_0,X_1)$ and $(Y_0,Y_1)$ form an $(m^2c)$-uniform relative Calder\'on-Mityagin pair and completes the proof.
\end{proof}

Next we establish similar results for $(L^1,(L^\infty)^o)$.

\begin{theorem}\label{CLL} Let $n$ be a constant satisfying Definition {\rm \ref{Ndef}}. Then,
\begin{enumerate}[leftmargin=1.85em, label=\rm{(\alph*)}]
\item\label{ECC} $(L^1,(L^\infty)^o)$ is an exact Calder\'on-Mityagin couple.
\item\label{ERC} $(L^1,L^\infty)$ and $(L^1,(L^\infty)^o)$ form an exact relative Calder\'on-Mityagin pair.
\item\label{URC} $(L^1,(L^\infty)^o)$ and $(L^1,L^\infty)$ form an $n$-uniform relative Calder\'on-Mityagin pair.
\item\label{notERC} $(L^1,(L^\infty)^o)$ and $(L^1,L^\infty)$ do not form an exact relative Calder\'on-Mityagin pair.
\end{enumerate}
\end{theorem}
\begin{proof} Statement \ref{ECC} is Theorem 3.8 of \cite{MS}. For \ref{ERC} and \ref{URC}, operator names refer to those constructed in Theorem \ref{Op}. Refer to the diagram (\ref{schema}).

To prove \ref{ERC}, suppose $h\in L^1+L^\infty$ and $g\in L^1+(L^\infty)^o$ satisfy the inequality $K(t,g;L^1,(L^\infty)^o)\le K(t,h;L^1,L^\infty)$ for $t>0$, and let $T=U_3IV_4$. All three have norm at most $1$ so $T:(L^1,L^\infty)\to(L^1,(L^\infty)^o)$ has norm at most $1$ as well. Also, $Th=U_3IV_4h=U_3Ig^o=U_3g^o=g$.

To prove \ref{URC}, suppose $f\in L^1+(L^\infty)^o$ and $h\in L^1+L^\infty$ satisfy $K(t,h;L^1,L^\infty)\le K(t,f;L^1,(L^\infty)^o)$ for $t>0$, and let $T=V_3N$, where $N$ comes from Definition \ref{Ndef}. The operator $V_3$ has norm at most one and $N$ has norm at most $n$. Thus $T:(L^1,(L^\infty)^o)\to(L^1,L^\infty)$ has norm at most $n$. Also, $Tf=V_3Nf=V_3f^o=h$.

Lemma \ref{exN} proves \ref{notERC}.
\end{proof}

\begin{theorem}\label{iff} If $X\in\Int_1(L^1,L^\infty)$ then $ X^o\in\Int_1(L^1,(L^\infty)^o)$. Conversely, if $Y$ is in $\Int_1(L^1,(L^\infty)^o)$ then $Y= X^o$, with identical norms, for some space $X$ in $\Int_1(L^1,L^\infty)$.
\end{theorem}
\begin{proof} If $X\in\Int_1(L^1,L^\infty)$ we apply the theorems of Calder\'on and Brudny\u\i-Krugljak to get a parameter $\Phi$ of the $K$-method such that $X=(L^1,L^\infty)_{\Phi}$, with identical norms. By Proposition \ref{Kfnls}, for any $f\in L^1+(L^\infty)^o$, the five statements: $f\in  X^o$, $ f^o\in X$, $K(\cdot, f^o;L^1,L^\infty)\in\Phi$, $K(\cdot,f;L^1,(L^\infty)^o)\in\Phi$ and $f\in(L^1,(L^\infty)^o)_\Phi$ are all equivalent. Moreover,
\[
\|f\|_{ X^o}=\| f^o\|_X=\|K(\cdot, f^o;L^1,L^\infty)\|_\Phi
=\|K(\cdot,f;L^1,(L^\infty)^o)\|_\Phi=\|f\|_{(L^1,(L^\infty)^o)_\Phi}.
\]
This shows $ X^o=(L^1,(L^\infty)^o)_\Phi$, with identical norms, so $ X^o\in \Int_1(L^1,(L^\infty)^o)$. 

On the other hand, if $Y\in\Int_1(L^1,(L^\infty)^o)$ then by Theorem \ref{CLL}\ref{ECC}, the Brudny\u\i-Krugljak theorem, and \cite[Corollary 3.9]{MS} (which shows that the $K$-divisibility constant of $(L^1,(L^\infty)^o)$ is $1$) there exists a parameter $\Phi$ of the $K$-method such that $Y=(L^1,(L^\infty)^o)_{\Phi}$, with identical norms. Let $X=(L^1,L^\infty)_{\Phi}\in\Int_1(L^1,L^\infty)$. Arguing as above we see that $ X^o=(L^1,(L^\infty)^o)_\Phi=Y$, with identical norms.
\end{proof}

\begin{lemma}\label{embed} Let $n$ be a constant satisfying Definition {\rm \ref{Ndef}}. Suppose $X_0$ and $X_1$ are in $\Int_1(L^1,L^\infty)$ and $\mathcal F$ is an exact interpolation functor. Set $X=\mathcal F(X_0,X_1)$. Then
\begin{enumerate}[leftmargin=1.85em, label=\rm{(\alph*)}]
\item\label{cons} There exist $\Phi_0$ and $\Phi_1$, parameters of the $K$-method, such that $X_j=(L^1,L^\infty)_{\Phi_j}$ and $ X_j^o=(L^1,(L^\infty)^o)_{\Phi_j}$, in each case with identical norms, for $j=0,1$.
\item\label{inX} $\mathcal F( X_0^o, X_1^o)\hookrightarrow X^o$ with norm at most $n$.
\item\label{Xin} $ X^o\hookrightarrow \mathcal F( X_0^o, X_1^o)$ with norm at most $1$.
\item\label{approxK} $ X_0^o+ X_1^o= (X_0+X_1)^o$, and for all $f$ in this space,
\[
(1/n)K(t, f^o;X_0,X_1)\le K(t,f; X_0^o, X_1^o)\le K(t, f^o;X_0,X_1),\quad t>0.
\]
\end{enumerate}
\end{lemma}
\begin{proof} 
The proof of part \ref{cons} is contained in the proof of Theorem \ref{iff}. As a consequence we have $ X_j^o\subseteq L^1+(L^\infty)^o$ for $j=0,1$. It follows that $F( X_0^o, X_1^o)\subseteq L^1+(L^\infty)^o$. Fix $f\in F( X_0^o, X_1^o)$. By Definition \ref{Ndef} there exists a linear operator $N:(L^1,(L^\infty)^o)\to(L^1,L^\infty)$, of norm at most $n$, such that $Nf=f^o$. By real interpolation with parameter $\Phi_j$, $N: X_j^o\to X_j$ with norm at most $n$, for $j=0,1$. Applying the exact functor $\mathcal F$ shows $N:\mathcal F( X_0^o, X_1^o)\to X$, with norm at most $n$. Since $f^o=Nf\in X$ we have $f\in  X^o$ and $\|f\|_{ X^o}=\|f^o\|_X=\|Nf\|_X\le n\|f\|_{\mathcal F( X_0^o, X_1^o)}$. This proves \ref{inX}.

Now let $g\in X^o$. Then $g^o\in X=\mathcal F(X_0,X_1)\subseteq L^1+L^\infty$ so $g\in L^1+(L^\infty)^o$. Proposition \ref{Kfnls} shows that $K(t, g^o;L^1,L^\infty)=K(t,g;L^1,(L^\infty)^o)$ so Theorem \ref{Op} provides operators $I$ and $U_3$ such that $U_3I:(L^1,L^\infty)\to(L^1,(L^\infty)^o)$, with norm at most $1$, and $U_3I g^o=g$. By real interpolation, $U_3I:X_j\to (L^1,(L^\infty)^o)_{\Phi_j}= X_j^o$, with norm at most $1$, for $j=0,1$. Applying $\mathcal F$, we get, $U_3I:X\to\mathcal F( X_0^o, X_1^o)$, with norm at most $1$. In particular, $\|g\|_{\mathcal F( X_0^o, X_1^o)}=\|U_3I g^o\|_{\mathcal F( X_0^o, X_1^o)}\le\|g^o\|_X=\|g\|_{ X^o}$. This proves \ref{Xin}.

For part \ref{approxK} we apply \ref{inX} and \ref{Xin} to the family of exact functors used in Lemma \ref{2embed}; $\Sigma_t(Z_0,Z_1)=Z_0+Z_1$ with  norm $\|f\|_{\Sigma_t(Z_0,Z_1)}=K(t,f;Z_0,Z_1)$. The embeddings of \ref{inX} and \ref{Xin} show that $ X_0^o+ X_1^o= (X_0+X_1)^o$ and, for all $f$ in this space, the estimates of the constants in \ref{inX} and \ref{Xin} show
\[
K(t, f^o;X_0,X_1)=\| f^o\|_{\Sigma_t(X_0,X_1)}\le n\|f\|_{\Sigma_t( {X_0}, X_1^o)}=nK(t,f; X_0^o, X_1^o),
\]
and
\[
K(t,f; X_0^o, X_1^o)=\|f\|_{\Sigma_t( X_0^o, X_1^o)}\le \| f^o\|_{\Sigma_t(X_0,X_1)}=K(t, f^o;X_0,X_1).\qedhere
\]
\end{proof}

When considering couples of spaces in $\Int_1(L^1,(L^\infty)^o)$, as we do in the converse below, there is no loss in generality in considering only couples of the form $( X_0^o, X_1^o)$ for $X_0,X_1\in \Int_1(L^1,L^\infty)$. This is a consequence of Theorem \ref{iff}. 
\begin{theorem}\label{CC} Let $n$ be a constant satisfying Definition {\rm \ref{Ndef}} and suppose that $X_0$, $X_1$, $Y_0$, and $Y_1$ are all in $\Int_1(L^1,L^\infty)$. If $(X_0,X_1)$ and $(Y_0,Y_1)$ form a $c$-uniform relative Calder\'on-Mityagin pair then $( X_0^o, X_1^o)$ and $( Y_0^o, Y_1^o)$ form an $(n^2c)$-uniform relative Calder\'on-Mityagin pair. Conversely, if $( X_0^o, X_1^o)$ and $( Y_0^o, Y_1^o)$ form a $c$-uniform relative Calder\'on-Mityagin pair then $(X_0,X_1)$ and $(Y_0,Y_1)$ form an $(n^2c)$-uniform relative Calder\'on-Mityagin pair.
\end{theorem}
\begin{proof} Suppose $(X_0,X_1)$ and $(Y_0,Y_1)$ form a $c$-uniform relative Calder\'on-Mityagin pair, $f\in X_0^o+ X_1^o$, $g\in Y_0^o+ Y_1^o$, and $K(t,g; Y_0^o, Y_1^o)\le K(t,f; X_0^o, X_1^o)$ for $t>0$. By Lemma \ref{embed}\ref{approxK} and the homogeneity of the $K$-functional,
\[
K(t, g^o;Y_0,Y_1)\le nK(t,g; Y_0^o, Y_1^o)
\le nK(t,f; X_0^o, X_1^o)\le K(t,n f^o;X_0,X_1).
\]
By hypothesis, there exists a linear operator $T:(X_0,X_1)\to(Y_0,Y_1)$, of norm at most $c$, such that $T(n f)= g^o$. Since $f\in X_0^o+ X_1^o$, $ f^o\in X_0+X_1\subseteq L^1+L^\infty$. The product $U_3I:(L^1,L^\infty)\to(L^1,(L^\infty)^o)$, from Theorem \ref{Op}, has norm at most $1$ and takes $ g^o$ to $g$. The operator $N:(L^1,(L^\infty)^o)\to(L^1,L^\infty)$, from Definition \ref{Ndef}, has norm at most $n$ and takes $f$ to $ f^o$. By Lemma \ref{embed}\ref{cons}, $U_3I:(Y_0,Y_1)\to( Y_0^o, Y_1^o)$ has norm at most $1$, and $N:( X_0^o, X_1^o)\to(X_0,X_1)$ has norm at most $n$. It follows that
\[
\includegraphics{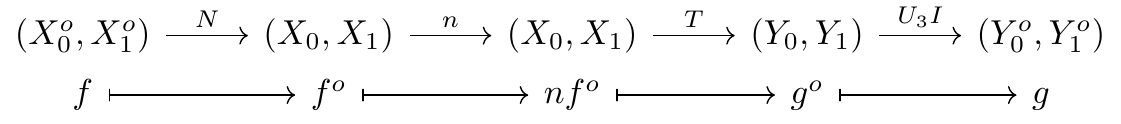}
%\begin{tikzcd}[row sep=.1ex]
%( X_0^o, X_1^o)\arrow[r,"N"]
%&(X_0,X_1)\arrow[r,"n"]
%&(X_0,X_1)\arrow[r,"T"]
%&(Y_0,Y_1)\arrow[r,"U_3I"]
%&( Y_0^o, Y_1^o)\\
%f\arrow[r,mapsto]
%& f^o\arrow[r,mapsto]
%&n f^o\arrow[r,mapsto]
%& g^o\arrow[r,mapsto]
%&g
%\end{tikzcd}
\]
where $n$ is used to denote multiplication by $n$. That is,
the operator $\bar T=U_3ITnN$ maps $( X_0^o, X_1^o)$ to $( Y_0^o, Y_1^o)$, with norm at most $n^2c$, and $\bar Tf=g$. We conclude that $( X_0^o, X_1^o)$ and $( Y_0^o, Y_1^o)$ form an $(n^2c)$-uniform relative Calder\'on-Mityagin pair.

Conversely, suppose $( X_0^o, X_1^o)$ and $( Y_0^o, Y_1^o)$ form a $c$-uniform relative Calder\'on-Mityagin pair, $f\in X_0+X_1$, $g\in Y_0+Y_1$ and $K(t,g;Y_0,Y_1)\le K(t,f;X_0,X_1)$. Since $g$ and $g^*$ are equimeasurable, and $f$ and $f^*$ are equimeasurable,
\[
K(t,g^*;Y_0,Y_1)=K(t,g;Y_0,Y_1)\le K(t,f;X_0,X_1)=K(t,f^*;X_0,X_1).
\]
But $g^*= (g^*)^o$ and $f^*= (f^*)^o$, so by Lemma \ref{embed}\ref{approxK} and the homogeneity of the $K$-functional,
\[
K(t,g^*; Y_0^o, Y_1^o)
\le K(t, (g^*)^o;Y_0,Y_1)
\le K(t, (f^*)^o;X_0,X_1)
\le K(t,nf^*; X_0^o, X_1^o).
\]
By the hypothesis, there exists a $T:( X_0^o, X_1^o)\to( Y_0^o, Y_1^o)$, with norm at most $c$ such that $T(nf^*)=g^*$.

Since $K(t,f^*;L^1,L^\infty)\le K(t,f;L^1,L^\infty)$ and $K(t,g;L^1,L^\infty)\le K(t,g^*;L^1,L^\infty)$, Calder\'on's theorem provides $V_f, V_g:(L^1,L^\infty)\to(L^1,L^\infty)$, each with norm at most $1$, such that $V_ff=f^*$ and $V_gg^*=g$. Exact interpolation shows that $V_f:(X_0,X_1)\to(X_0,X_1)$ and $V_g:(Y_0,Y_1)\to(Y_0,Y_1)$, each with norm at most $1$.

By Definition \ref{Ndef} there exists an $N:(L^1,(L^\infty)^o)\to(L^1,L^\infty)$, with norm at most $n$, such that $Ng^*= {(g^*)^o}$, that is, $Ng^*=g^*$. In Theorem \ref{Op}\ref{I} it was observed that the identity operator $I$ maps $(L^1,L^\infty)$ to $(L^1,(L^\infty)^o)$, with norm at most $1$. By exact interpolation, $I:(X_0,X_1)\to( X_0^o,  X_1^o)$, with norm at most $1$, and $N:( Y_0^o, Y_1^o)\to(Y_0,Y_1)$, with norm at most $n$.  Now,
\[
\includegraphics{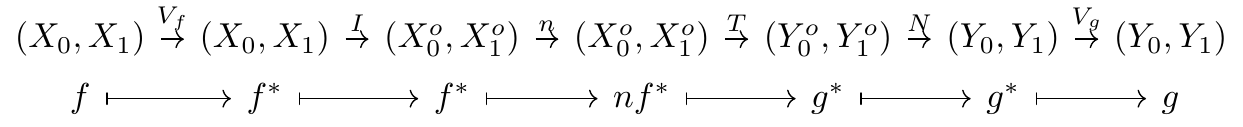}
%\begin{tikzcd}[row sep=.1ex, column sep=.7em]
%(X_0,X_1)\arrow[r,"V_f"]
%&(X_0,X_1)\arrow[r,"I"]
%&( X_0^o, X_1^o)\arrow[r,"n"]
%&( X_0^o, X_1^o)\arrow[r,"T"]
%&( Y_0^o, Y_1^o)\arrow[r,"N"]
%&(Y_0,Y_1)\arrow[r,"V_g"]
%&(Y_0,Y_1)\\
%f\arrow[r,mapsto]
%&f^*\arrow[r,mapsto]
%&f^*\arrow[r,mapsto]
%&nf^*\arrow[r,mapsto]
%&g^*\arrow[r,mapsto]
%&g^*\arrow[r,mapsto]
%&g
%\end{tikzcd}
\]
That is, with $\bar T=V_gNTnIV_f$, $\bar T:(X_0,X_1)\to(Y_0,Y_1)$, with norm at most $n^2c$, and $\bar Tf=g$. This shows that $(X_0,X_1)$ and $(Y_0,Y_1)$ form an $(n^2c)$-uniform relative Calder\'on-Mityagin pair and completes the proof.
\end{proof}

It is interesting to observe the special case $(X_0,X_1)=(Y_0,Y_1)$ of Theorems \ref{2CC} and \ref{CC}, stated here without careful tracking of constants.
\begin{corollary}\label{allCC} For $X_0,X_1\in\Int_1(L^1,L^\infty)$ the following are equivalent:
\begin{enumerate}[leftmargin=1.85em, label=\rm{(\alph*)}]
\item\label{ri} $(X_0,X_1)$ is a Calder\'on-Mityagin couple.
\item\label{ldm} $(\2X_0,\2X_1)$ is a Calder\'on-Mityagin couple.
\item\label{lev} $( X_0^o, X_1^o)$ is a Calder\'on-Mityagin couple.
\end{enumerate}
\end{corollary}
\section{Extension to General Measures}\label{genMeas}

There are natural analogues of the operations $f\mapsto \1f$ and $f\mapsto f^o$ in spaces of $\lambda$-measurable functions on $\mathbb R$, provided $\lambda(-\infty,x]<\infty$ for $x\in\mathbb R$. The corresponding spaces $\2X$ and $ X^o$ enjoy the same close connection to rearrangement invariant spaces that they do in the case of Lebesgue measure on $(0,\infty)$. In particular, most of the results of Sections \ref{operators} and \ref{CCaRCP} remain valid for functions on a finite interval, for sequence spaces.

Our approach to proving these results is to identify the spaces of $\lambda$-measurable functions with complemented subspaces of the corresponding spaces over $(0,\infty)$ and then apply the previous results.

To begin, let us introduce the analogues of the least decreasing majorant and the level function. Let $\lambda$ be a measure on $\mathbb R$ for which the Borel subsets are measurable and $\Lambda(x)\equiv\lambda(-\infty,x]<\infty$ for all $x\in \mathbb R$. The least decreasing majorant of a $\lambda$-measurable $f$ is
\[
\1f(x)=\essup_\lambda\{|f(y)|:y\ge x\}.
\]
If $X$ is a Banach function space of $\lambda$-measurable functions, let $\2X$ be the space of all $\lambda$-measurable functions such that $\1f\in X$, equipped with the norm $\|f\|_{\2X}=\|\1f\|_X$.

We say a non-negative, $\lambda$-measurable function $F$ on $\mathbb R$ is \emph{$\lambda$-concave} provided
\[
(\Lambda(b)-\Lambda(x))(F(x)-F(a))\ge(F(b)-F(x))(\Lambda(x)-\Lambda(a))
\]
whenever $a\le x\le b$. (Note that when $\lambda$ is Lebesgue measure on $(0,\infty)$ this agrees with the usual notion of concavity.) The \emph{level function} $ f^o$ of $f$ is the Radon-Nikodym derivative, with respect to $\lambda$, of the least $\lambda$-concave majorant of $\int_{(-\infty,x)}f\,d\lambda$. As in the case of Lebesgue measure on $(0,\infty)$, it may be necessary to extend this definition by monotonicity in case the $\lambda$-concave majorant fails to be finite. If $X\in\Int_1(L^1_\lambda, L^\infty_\lambda)$, let $ X^o$ be the space of all $\lambda$-measurable functions such that $ f^o\in X$, equipped with the norm $\|f\|_{ X^o}=\| f^o\|_X$.

Notice that the operations $f\mapsto\1f$ and $f\mapsto f^o$, as well as the spaces $\2X$ and $ X^o$, depend on the measure $\lambda$. To avoid confusion with these ``overloaded'' operators, we will be careful to associate each function with a particular measure before speaking of its least decreasing majorant or its level function.

We again observe that $ ({L^1_\lambda})^o=L^1_\lambda$ and $\2{L^\infty_\lambda}={L^\infty_\lambda}$ with identical norms. Also, $\|f\|_{{L^1_\lambda}}\le\|f\|_{\2{L^1_\lambda}}$ and
\[
\|f\|_{(L^\infty_\lambda)^o}=\sup_{x\in\mathbb R}\frac1{\Lambda(x)}\int_{(-\infty,x]}|f|\,d\lambda\le\|f\|_{{L^\infty_\lambda}}.
\]

The \emph{method of retracts} (see pages 54f in \cite{BS2}) embeds a $\sigma$-finite measure space into a non-atomic one and exploits the close connection between spaces of measurable functions with respect to these two measures. When the original measure $\lambda$ is on $\mathbb R$ and satisfies $\lambda(-\infty,x]<\infty$ for each $x$, this method takes a particularly simple and powerful form that, in particular, preserves decreasing functions. This construction, from Section 4 of \cite{MS}, gives an order-preserving, measurable transformation from $(\mathbb
R,\lambda)$ into a subspace of $(0,\infty)$ with Lebesgue measure. Define $\Omega$, $\varphi$, $E_\lambda$, and $A_\lambda$ by
\begin{eqnarray*}
\Omega&=&\{t>0:t\le\Lambda(y)\text{
for some }y\in\mathbb R\},\\
\varphi(t)&=&\inf\{y:t\le \Lambda(y)\},\quad t\in\Omega,\\
 E_\lambda f&=&(f\circ\varphi)\chi_\Omega,\quad\text{and}\\
A_\lambda h(t)&=&\begin{cases}\frac{\chi_\Omega(t)}{|I|}\int_Ih,&t\in I\in\mathcal I_\lambda\\
\chi_\Omega(t)h(t),&t\notin\cup_I\in\mathcal I_\lambda.\end{cases}
\end{eqnarray*}
Here $\mathcal I_\lambda$ denotes the collection of all non-empty intervals $(\Lambda(y-),\Lambda(y)]\subseteq\Omega$ for $y\in \mathbb R$.

The map $ E_\lambda$ takes $\lambda$-measurable functions to Lebesgue measurable functions, respecting the level function and least decreasing majorant constructions. Theorem 5.3 of \cite{MS} shows that $ ( E_\lambda f)^o= E_\lambda  f^o$ on $\Omega$. Note that $ f^o$ is the level function of $f$ with respect to $\lambda$, while $ ( E_\lambda f)^o$ is the level function of $ E_\lambda f$ with respect to Lebesgue measure. The corresponding result for the least decreasing majorant is
$\2{ E_\lambda f}= E_\lambda\1f$ on $\Omega$. To see this, let $t\in\Omega$. Then
\[
\2{ E_\lambda f}(t)=\essup\{f(\varphi(s)):t\le s\in\Omega\}
\le\essup_\lambda\{f(x):\varphi(t)\le x\}= E_\lambda\1f(t)
\]
and, since $\varphi(t)\le x$ if and only if $t\le \Lambda(x)$ and $x=\varphi(\Lambda(x))$ $\lambda$-almost everywhere (see (5) and Lemma 4.1 of \cite{MS}),
\begin{align*}
 E_\lambda\1f(t)=&\essup_\lambda\{f(\varphi(\Lambda(x))):t\le \Lambda(x)\}\\
\le& \essup\{f(\varphi(s)):t\le s\in\Omega\}=\2{ E_\lambda f}(t).
\end{align*}

The maps $ E_\lambda$ and $A_\lambda$ work together to identify function spaces of $\lambda$-measurable functions with subspaces of the corresponding spaces of Lebesgue measurable functions. This is because $A_\lambda$ is a projection onto the range of $ E_\lambda$.
\begin{proposition} The map $A_\lambda$ is a projection defined on $L^1+(L^\infty)^o$ and is a contraction on $L^1$, $(L^\infty)^o$, $\2{L^1}$, and $L^\infty$. Also, $A_\lambda E_\lambda= E_\lambda$, and the maps $ E_\lambda:{L^1_\lambda}\leftrightarrow A_\lambda(L^1)$,
$ E_\lambda:(L^\infty_\lambda)^o\leftrightarrow A_\lambda((L^\infty)^o)$,
$ E_\lambda:\2{L^1_\lambda}\leftrightarrow A_\lambda(\2{L^1})$, and
$ E_\lambda:{L^\infty_\lambda}\leftrightarrow A_\lambda(L^\infty)$, are bijective isometries.
\end{proposition}
\begin{proof} It is easy to see that $A_\lambda$ is a projection defined on  $L^1+(L^\infty)^o$ and that it is a contraction on $L^\infty$. Corollary 3.2 of \cite{MS} shows that $A_\lambda$ is a contraction on $L^1$ and $(L^\infty)^o$. Since the averages in $A_\lambda$ are all taken over intervals, $A_\lambda$ maps non-negative decreasing functions to non-negative decreasing functions. Therefore, $\2{A_\lambda f}\le\2{A_\lambda\1f}=A_\lambda\1f$, so $\|A_\lambda f\|_{\2{L^1}}\le\|A_\lambda\1f\|_{L^1}\le\|\1f\|_{L^1}=\|f\|_{\2{L^1}}$. Thus, $A_\lambda$ is a contraction on $\2{L^1}$.

Lemma 4.4 of \cite{MS} shows $A_\lambda E_\lambda= E_\lambda$ and that $ E_\lambda:{L^1_\lambda}\leftrightarrow A_\lambda(L^1)$ and $ E_\lambda:(L^\infty_\lambda)^o\leftrightarrow A_\lambda((L^\infty)^o)$ are bijective isometries. Since $\2{L^1_\lambda}\subseteq {L^1_\lambda}$ and ${L^\infty_\lambda}\subseteq(L^\infty_\lambda)^o$ $ E_\lambda:\2{L^1_\lambda}\leftrightarrow A_\lambda(\2{L^1})$ and
$ E_\lambda:{L^\infty_\lambda}\leftrightarrow A_\lambda(L^\infty)$ are both bijections. It remains to show that they are isometric. For $f\in \2{L^1_\lambda}$,
\[
\| E_\lambda f\|_{\2{L^1}}=\|\2{ E_\lambda f}\|_{L^1}=\| E_\lambda \1f\|_{L^1}=\|\1f\|_{L^1_\lambda}=\|f\|_{\2{L^1_\lambda}}.
\]
It is an easy consequence of Lemma 4.2 of \cite{MS} that $f$ and $ E_\lambda f$ are equimeasurable. Therefore $\|E_\lambda f\|_{L^\infty}=\|f\|_{{L^\infty_\lambda}}$.
\end{proof}

\begin{corollary}\label{Ksubs}
\begin{enumerate}[leftmargin=1.85em, label=\rm{(\alph*)}]
\item\label{1oo} If $f\in {L^1_\lambda}+{L^\infty_\lambda}$ then $K(t,f;{L^1_\lambda},{L^\infty_\lambda})=K(t, E_\lambda f;L^1,L^\infty)$.
\item\label{21oo} If $f\in \2{L^1_\lambda}+{L^\infty_\lambda}$ then
\[
K(t,f;\2{L^1_\lambda},{L^\infty_\lambda})=K(t, E_\lambda f;\2{L^1},L^\infty)=K(t,\1f;{L^1_\lambda},{L^\infty_\lambda}).
\]
\item\label{31oo} If $f\in {L^1_\lambda}+(L^\infty_\lambda)^o$ then
\[
K(t,f;{L^1_\lambda},(L^\infty_\lambda)^o)=K(t, E_\lambda f;L^1,(L^\infty)^o)=K(t, f^o;{L^1_\lambda},{L^\infty_\lambda}).
\]
\end{enumerate}
\end{corollary}
\begin{proof} Fix $f\in {L^1_\lambda}+{L^\infty_\lambda}$. If $f=f_0+f_1$ with $f_0\in {L^1_\lambda}$ and $f_1\in {L^\infty_\lambda}$, then $ E_\lambda f= E_\lambda f_0+ E_\lambda f_1$ so
\[
K(t, E_\lambda f;L^1,L^\infty)\le\| E_\lambda f_0\|_{L^1}+t\| E_\lambda f_1\|_{L^\infty}
=\|f_0\|_{{L^1_\lambda}}+t\|f_1\|_{{L^\infty_\lambda}}.
\]
Take the infimum over all such $f_0$ and $f_1$ to get ``$\ge$'' in \ref{1oo}. 

For the reverse inequality, suppose $ E_\lambda f=h_0+h_1$ with $h_0\in L^1$ and $h_1\in L^\infty$, then
\[
f= E_\lambda^{-1}A_\lambda E_\lambda f= E_\lambda^{-1}A_\lambda h_0+ E_\lambda^{-1}A_\lambda h_1.
\]
But $ E_\lambda^{-1}A_\lambda$ has norm at most $1$ on both $L^1$ and $L^\infty$, so
\[
K(t,f;{L^1_\lambda},{L^\infty_\lambda})\le\| E_\lambda^{-1}A_\lambda h_0\|_{{L^1_\lambda}}+t\| E_\lambda^{-1}A_\lambda h_1\|_{{L^\infty_\lambda}}\le\|h_0\|_{L^1}+t\|h_1\|_{L^\infty}.
\]
Take the infimum over all such $h_0$ and $h_1$ to get ``$\le$'' in \ref{1oo}.

With ${L^1_\lambda}$ and $L^1$ replaced by $\2{L^1_\lambda}$ and $\2{L^1}$, the same argument proves the first equation in \ref{21oo}. The other follows from Proposition \ref{Kfnls} and part \ref{1oo}.

With ${L^\infty_\lambda}$ and $L^\infty$ replaced by $(L^\infty_\lambda)^o$ and $(L^\infty)^o$, the same argument also proves the first equation in \ref{31oo}. The other follows from Proposition \ref{Kfnls} and part \ref{1oo}.
\end{proof}

Here is our extension of Theorems \ref{C2LL} and \ref{CLL} to general measures. Interestingly, they generalize to the case of two measures as easily as to one.
\begin{theorem}\label{LLgen} Let $m$ and $n$ be constants satisfying Definitions {\rm \ref{Mdef}} and {\rm \ref{Ndef}}, respectively. Then,
\begin{enumerate}[leftmargin=1.85em, label=\rm{(\alph*)}]
\item\label{LL2ECC} $(\2{L^1_\mu},{L^\infty_\mu})$ is an exact Calder\'on-Mityagin couple.
\item\label{LL2ECCP} $(\2{L^1_\mu},{L^\infty_\mu})$ and $(\2{L^1_\nu},{L^\infty_\nu})$ form an exact relative Calder\'on-Mityagin pair.
\item\label{LL2ECP} $(\2{L^1_\mu},{L^\infty_\mu})$ and $({L^1_\nu},{L^\infty_\nu})$ form an exact relative Calder\'on-Mityagin pair.
\item\label{LL2UCP} $({L^1_\mu},{L^\infty_\mu})$ and $(\2{L^1_\nu},{L^\infty_\nu})$ form an $m$-uniform relative Calder\'on-Mityagin pair.
\item\label{LLECC} $({L^1_\mu},(L^\infty_\mu)^o)$ is an exact Calder\'on-Mityagin couple.
\item\label{LLECCP} $({L^1_\mu},(L^\infty_\mu)^o)$ and $({L^1_\nu},(L^\infty_\nu)^o)$ form an exact relative Calder\'on-Mityagin pair.
\item\label{LLECP} $({L^1_\mu},{L^\infty_\mu})$ and $({L^1_\nu},(L^\infty_\nu)^o)$ form an exact relative Calder\'on-Mityagin pair.
\item\label{LLUCP} $({L^1_\mu},(L^\infty_\mu)^o)$ and $({L^1_\nu},{L^\infty_\nu})$ form an $n$-uniform relative Calder\'on-Mityagin pair.
\end{enumerate}
\end{theorem}
\begin{proof} First note that \ref{LL2ECC} and \ref{LLECC} are just the case $\nu=\mu$ of \ref{LL2ECCP} and \ref{LLECCP}, respectively. We prove part \ref{LL2UCP}; the remaining parts may be proved in a similar fashion.

Suppose $f\in {L^1_\mu}+{L^\infty_\mu}$ and $g\in\2{L^1_\nu}+{L^\infty_\nu}$ satisfy $K(t,g;\2{L^1_\nu},{L^\infty_\nu})\le K(t,f;{L^1_\mu},{L^\infty_\mu})$. Then $K(t, E_\nu g;\2{L^1},L^\infty)\le K(t, E_\mu f;L^1,L^\infty)$ so, by Theorem \ref{C2LL}\ref{2URC},
there exists a linear operator $T:(L^1,L^\infty)\to(\2{L^1},L^\infty)$, of norm at most $m$, such that $T( E_\mu f)= E_\nu g$. The operator $\bar T= E_\nu^{-1}A_\nu T E_\mu:({L^1_\mu},{L^\infty_\mu})\to(\2{L^1_\nu},{L^\infty_\nu})$ has norm at most $m$ and
\[
\bar Tf= E_\nu^{-1}A_\nu T E_\mu f= E_\nu^{-1}A_\nu E_\nu g=g.\qedhere
\]
\end{proof}

We leave it to the reader to verify that Lemmas \ref{Kdiv}, \ref{2embed} and \ref{embed} and Theorems \ref{2Op}, \ref{Op}, \ref{2iff} and \ref{iff} hold with $L^1$, $L^\infty$, $\2{L^1}$, and $(L^\infty)^o$ replaced by ${L^1_\lambda}$, ${L^\infty_\lambda}$, $\2{L^1_\lambda}$, and $(L^\infty_\lambda)^o$, respectively. Essentially one replaces each operator $O$ by $ E_\lambda^{-1}A_\lambda O E_\lambda$, applying Corollary \ref{Ksubs} when needed. Note that Corollary 3.9 of \cite{MS} was extended to general measures in Corollary 4.7 of the same paper.

The first statement in each of Theorems \ref{2CC} and \ref{CC} may be extended to the general measure case by the same procedure (and this includes extensions of the implications \ref{ri}$\implies$\ref{ldm} and \ref{ri}$\implies$\ref{lev} of Corollary \ref{allCC}). However, it appears likely that the second statement of each will fail for certain measures. The example below shows that the couple $({L^1_\lambda}, {L^\infty_\lambda})$ may have a rich collection of exact interpolation spaces, while both $(\2{L^1_\lambda},{L^\infty_\lambda})$ and $({L^1_\lambda},(L^\infty_\lambda)^o)$ become trivial.

\begin{example} Let $\lambda$ be the probability measure on $\mathbb R$ consisting of atoms of weight $2^{-k}$ at $k$, for $k\in\mathbb N$. If $X\in \Int_1(L^1_\lambda,L^\infty_\lambda)$ then $\2X={L^\infty_\lambda}$ with equivalent norms and $ X^o={L^1_\lambda}$ with equivalent norms.
\end{example}
\begin{proof} The constant function $1$ is in ${L^1_\lambda}\cap{L^\infty_\lambda}$ and hence in $X$. Thus $0<\|1\|_X<\infty$. A calculation shows that
\[
1^{**}=\min(1,\tfrac1x)\le2\min(1,\tfrac1{2x})=2\chi_{\{1\}}^{**}.
\]
Since $X\in \Int_1(L^1_\lambda,L^\infty_\lambda)$, $\|1\|_X\le 2\|\chi_{\{1\}}\|_X$. For any $\lambda$-measurable $f$, $\|f\|_{{L^\infty_\lambda}}\chi_{\{1\}}\le\1f\le\|f\|_{{L^\infty_\lambda}}$ and therefore
\[
\|f\|_{\2X}=\|\1f\|_X\le\|f\|_{{L^\infty_\lambda}}\|1\|_X\le2\|f\|_{{L^\infty_\lambda}}\|\chi_{\{1\}}\|_X\le2\|\1f\|_X=2\|f\|_{\2X}.
\]
This proves the first statement. Since $ f^o(1)\chi_{\{1\}}\le f^o\le f^o(1)$, 
\[
\|f\|_{ X^o}=\| f^o\|_X\le  f^o(1)\|1\|_X\le2 f^o(1)\|\chi_{\{1\}}\|_X\le2\| f^o\|_X=2\|f\|_{ X^o}.
\]
Combining this with $\|f\|_{{L^1_\lambda}}=\| f^o\|_{{L^1_\lambda}}\le f^o(1)\le2\| f^o\|_{{L^1_\lambda}}$ proves the second statement.
\end{proof}

\end{document}